\documentclass[12pt]{amsart}
\usepackage{amssymb} 

\newtheorem{thm}{Theorem}[section]
\newtheorem{cor}[thm]{Corollary}
\newtheorem{lem}[thm]{Lemma}
\newtheorem{prop}[thm]{Proposition}
\newtheorem{exam}[thm]{Example}

\theoremstyle{definition}
\newtheorem{defn}[thm]{Definition}
\newtheorem{rem}[thm]{Remark}

\newtheorem{que}[thm]{Question}

\numberwithin{equation}{section}

\usepackage[colorlinks]{hyperref}

\begin{document}
\title[A generalization of the order elements with...]
{A generalization of the order of elements with   some applications}%
\author[  M.Amiri  ]{  Mohsen Amiri}%
\address{ Departamento de  Matem\'{a}tica, Universidade Federal do Amazonas.}%
\email{mohsen@ufam.edu.br}%
\email{}
\thanks{}
\thanks{}

\begin{abstract}
In this paper,  we  give some   generalizations   the concept of element order and we study  some of the properties of these  generalized order.  In particular,  with using this generalization we  derive two solvability criteria.
\end{abstract}

\maketitle


\section{\bf Introduction}

For a finite group, the order of group and the element orders are two of the most important basic
concepts. Let $x$  be an element of a finite group $G$.  A natural question arises: {\it what can we say about the order of the elements of  the coset $xN$ when $N$ is not a normal subgroup}?
In general, it  is hard to get much information about the order of the elements of the coset.
For this reason  we    give a generalization of   the order of elements that  just  depends   on subgroups of $G$.  With using this generalization we  derive two solvability criteria.
 In what follows, we adopt the notation established in the Isaacs' book on finite groups  \cite{I}.


\section{Generalized element orders}\label{sec2}
In this section we shall generalize the concept of element order by defining two types of element order  $``o"$  and $``E"$.  We shall also study their basic properties.

First we recall the following well-known properties of the usual order of an element $x$ in a group $G$.
 \begin{thm}
Let $G$  be a finite    group,  $N$ a normal subgroup of $G$ and $x\in G$. Then the following three properties concerning the order of elements are satisfied:

(i) $o(xN)\mid [G:N]$,

(ii) $o(x^gN)=o(xN)$ for all $ g\in G$,

(iii) $o(xnN)=o(xN)$ for all $n\in N$.
\end{thm}
A natural question arises: {\it what can we say about the order of the elements of  the coset $xN$ when $N$ is not a normal subgroup}?
In general, it  is hard to get much information about the order of the elements of the coset.
For this reason  we    give a generalization of   the order of elements that  depends   on subgroups of $G$. To do so, we begin with the following observation. For each subgroup $N\leq G$ and element  $x\in G$ we write $n(x, N)$ for the smallest positive  integer such that $x^{n(x, N)}\in N$.
It is clear that $\langle x^{n(x, N)}\rangle=\langle x\rangle \cap N \leq \langle x\rangle,$ and
$n(x, N)= [\langle x\rangle:\langle x\rangle \cap N]$. As the number $n(x,N)$ will be used in defining the generalized element orders, we first study some of its properties.
\begin{lem}\label{1} Let $N\leq M \leq G$ be a sequence of subgroups of $G$ such that $[G:N]<\infty$, $x, y, z\in G$ and $m$ a positive integer. Then, the following are true

(i) $n(x^m, N)=\dfrac{n(x, N)}{gcd(n(x, N), m)}$.

(ii) If $x=yz=zy$, then $n(yz, N)\mid lcm(n(y, N), n(z, N))$.
Also, $n(yz, N)= n(y, N)n(z, N)$  if and only if $gcd(n(y, N), n(z, N))=1$.

(iii)
$n(x, N)=n(x, M) n(x^{n(x, M)}, N)$.
 \end{lem}
\begin{proof}

(i) First we prove that if $n(x^m, N)=1$, then $n(x, N)\mid m$.
There are positive integers $r$ and $q$ such that $m=qn(x,N)+r$ with $0\leq r<n(x,N)$.
Since $x^m$ and $x^{-q\cdot n(x, N)}\in N$, we have $x^r\in N$. Hence $r=0.$

Let $d=gcm(n(x, N), m)$. It follows from $(x^m)^{\frac{n(x, N)}{d}}=(x^{\frac{m}{d}})^{n(x, N)}\in N$
that $n(x^m, N)\mid \frac{n(x, N)}{d}$.
Since $(x^m)^{n(x^m, N)}\in N$, we conclude that
$n(x, N)\mid m\cdot n(x^m, N)$.
Therefore $\frac{n(x, N)}{d}\mid \frac{m}{d} n(x^m, N)$.
Since $gcd\Big( \frac{n(x, N)}{d}, \frac{m}{d}\Big)=1$, we have
 $\frac{n(x, N)}{d}\mid  n(x^m, N)$.

(ii) Let $m=lcm(n(y, N), n(z, N))$, and let $x=yz$. We compute
$$x^m=(z^{n(z, N)})^{\frac{n(y, N)}{gcd(n(y, N), n(z, N))}}(y^{n(y, N)})^{\frac{n(z, N)}{gcd(n(y, N), n(z, N))}}\in N,$$ which means that $n(x, N)\mid m$.

Now, suppose  that  $n(yz, N)=  n(y, N) n(z, N)$.  It follows from the first part that
 $n(x, N)\mid m$. Therefore  $gcd(n(y, N), n(z, N))=1$.
 
 So, let $gcd(n(y, N), n(z, N))=1$. We prove that $n(yz, N)=  n(y, N)  n(z, N)$.
 From the first part $n(x, N)\mid m$, so $n(x, N)\mid n(y, N)n(z, N)$.
 Since  $x^{n(x, N)}=y^{n(x, N)}z^{n(x, N)}\in N$, there exists $h\in N$, such that
  $y^{n(x, N)}=z^{-n(x, N)}h$. Since $zy=yz$, we have $z^{-n(x, N)}hz=zz^{-n(x, N)}h$, and so
  $zh=hz$.
 Since  $(y^{n(x, N)})^{n(y, N)}\in N$, we deduce that  $(z^{-n(x, N)})^{n(y, N)}h^{n(y, N)}\in N$, and so
 $(z^{-n(x, N)})^{n(y, N)}\in N$. It follows that
 $n(z, N)\mid n(x, N)n(y, N)$. As
  $gcd(n(y, N), n(z, N))=1$, we have
  $n(z, N)\mid n(x, N)$.  By a similar argument to the one
   above, we have
   $n(y, N)\mid n(x, N)$. Then
   $ n(y, N) n(z, N)\mid n(x, N)$. It follows that  $ n(x, N)= n(y, N)n(z, N)$.

(iii)
Let $z=x^{n(x, M)}$.  We have
$\langle x^{n(x, N)}\rangle=\langle z^{n(z, N)}\rangle$. It follows that
\begin{eqnarray*}
\frac{o(x)}{n(x, N)}&=&o( x^{n(x, N)})\\&=&o(z^{n(z, N)})\\&=&\frac{o(z)}{n(z, N)}\\&=&\frac{o(x)}{n(x, M) n(x^{n(x, M)}, N)}. 
\end{eqnarray*}
Hence $n(x, N)=n(x, M) n(z, N_t)=n(x, M) n(x^{n(x, M)}, N)$.
\end{proof}

We are now in a position to define the orders $``o"$, $``e"$ and $``E"$.
\begin{defn}\label{def23} Let $N\subseteq G$  such that $1\in N$. Let $x\in G$ such that $$ \{n(x^g,N): g\in G\}=\{n(x^{g_1}, N),...,n(x^{g_n},N)\}$$ where $n$ is a positive integer and $g_1,...,g_n\in G$.
 We define the order $``o"$ of the element $x$ with respect to the subset $N$ by
$$o(x, N)=gcd\Big( n(x^{g_1}, N), ... , n(x^{g_n}, N)\Big).$$

\end{defn}

Notice that Definition \ref{def23} also holds in the case when $G$ is an infinite group, for example,
whenever $N$ is a subgroup of $G$ and $[G:N]<\infty$.

\begin{defn}\label{def2}Consider a  finite group $G$. Let $N \subseteq M$  be   subsets of $G$.

We define the order $``e"$ of the element $x\in G$ with respect to the above subsets   by
$$e(x, M)=o(x, M)$$ \, and \,
 $$e(x, N)=e(x, M)\cdot   o(x^{e(x, M)}, N).$$
\end{defn}

\begin{defn}
Furthermore, if $N\leq M$  be   subgroups of $G$, then we define the order $``E"$ of the element $x\in G$ with respect to the above subgroups   by
$$E(x, M)=o(x, M)$$ \, and \,
 $$E(x, N)=E(x, M)\cdot  gcd\Big([M:N], o(x^{E(x, M)}, N)\Big).$$
\end{defn}

In the reminder of this section we will study some implications of these definitions. Among other things, we shall prove by the end of this section that when  $|G|<\infty$ and $N\leq G$   then
$o(x, N)=E(x, N)=e(x, N)$.   Notice that the latter equality of element orders shows that the definition of $E(x, N)$ only depends on $x$ and the subgroup $N$,   and order $x$ is equal  to     the order of  $x^g$ for all $g\in G$. Henceforth always assume that $G$ is a finite group and $N\leq M$  are   subgroups of $G$.

First we  show  that the $E$-order is indeed a generalization of the usual element order.
 
\begin{lem}\label{min} Let  $x\in G$ and let $p$ abe a prime divisor of $o(x)$ such that $gcd(o(x),p^{m+1})=p^m$. Let   $  o(x^{E(x, M)}, N) =p^{\alpha}v$ where
$gcd(p, v)=1$.
   Then
 $p^{\alpha}\mid [M:N]$.
\end{lem}
\begin{proof} Let $o(x)=p^mw $  where
$gcd(p,w)=1$. If $M=N$, then $\alpha=0$, and  so  the proof is clear.
So suppose that $M\neq N$.

 Let  $\frac{n(x^g, M)}{n(x^g,N)}=p^{\alpha_g}s_g$ where $gcd(p,s_g)=1$ for all $g\in G$. Let  $\alpha=\min\{\alpha_g : g\in G\}$. We may assume that $\alpha=\alpha_1$.    Let
 $y=x^{n(x,M)}$, and let $o(y)=p^d r$ where $p\nmid r$.  Since all Sylow $p$-subgroups of $M$ are conjugated,
there are $P\in Syl_p(M)$ and  $Q\in Syl_p(N)$ such that  $\langle y^{o(y)/p^d}\rangle Q^g\subseteq P$ for some $g\in M$.
Since $\alpha=\min\{\alpha_g : g\in G\}$, we have $|\langle y\rangle \cap Q^g|=p^{d-\alpha}$.  It follows that $|P|\geq |\langle y\rangle Q|= |Q| p^{\alpha}$ and hence $p^{\alpha}\mid [P:Q]$. It follows  from $[P:Q]\mid [M:N]$ that $p^{\alpha}\mid [M:N]$.
\end{proof}
\begin{cor}\label{eqal} Let  $x\in G$  be a finite    group.
Then $$E(x, N_t)=o(x,N_t).$$
\end{cor}

\begin{proof}Let $G=\{g_1,...,g_n\}$.
If $M=G$, then by definition $E(x,N)=o(x,N)$. So suppose that $M\neq N$. By Lemma \ref{min},
$E(x,N)=E(x,M)o(x^{E(x,M)},N)$. 
Clearly, 
 $$E(x, M)=gcd( n(x^{g_1}, M),..., n(x^{g_n}, M))\mid n(x^{g}, M),$$
for all $g\in G$. Let $g\in G$. Since $n(x^{g}, M)\mid n(x^{g}, N)$ and $E(x,M)\mid n(x^g,M)$, we have $E(x, M)\mid  n(x^{g}, N)$.  By Lemma \ref{1}(i),  $n((x^{g})^{E(x, M)}, N_t)=\frac{n(x^{g}, N)}{E(x, M)}$
 for all $g\in G$. Then
 \begin{eqnarray*}
E(x, N)&=&E(x, M)gcd( n((x^{g_1})^{E(x, M)}, N),..., n((x^{g_n})^{E(x, M)}, N))\\&=&
 E(x, M)(gcd( n(x^{g_1}, N),..., n(x^{g_n}, N))/E(x, M))\\&=&
 gcd( n(x^{g_1}, N),..., n(x^{g_n}, N))\\&=&o(x, N).
 \end{eqnarray*}
\end{proof}

 Let  $x\in G$ and $d$ a positive integer. Then we have $o(x^d)=\frac{o(x)}{gcd(d, o(x))}$. In what follows we will prove the same equation for the order of $x$ with respect to subgroup $N$.
 The following  number theory lemma is useful  to prove Lemma \ref{dd}.
\begin{lem}\label{numb}
Let $d,a_1,a_2,...,a_k$ be   positive integers with $k>1$.  If $gcd(d,gcd(a_1,...,a_k))=1$, then
$gcd(a_1,...,a_k)=gcd(\frac{a_1}{gcd(d,a_1)},...,\frac{a_k}{gcd(d,a_k)})$.
 \end{lem}
\begin{proof}
Let $m=gcd(a_1,...,a_k)$. Clearly, $gcd(\frac{a_1}{gcd(d,a_1)},...,\frac{a_k}{gcd(d,a_k)}))\mid m$.
Since $gcd(d,m)=1$, we have $gcd(gcd(d,a_1),m)=...=gcd(gcd(d,a_k),m)=1$.
Since $m\mid a_i$, we have $m\mid \frac{a_i}{gcd(d,a_i)}$ for all $i=1,2,...,k$.
It follows that $m\mid gcd(\frac{a_1}{gcd(d,a_1)},...,\frac{a_k}{gcd(d,a_k)})$.
\end{proof}
Let $d,a_1,a_2,...,a_k$ be   positive integers with $k>1$.
The following equality is a  direct  consequence of Lemma \ref{numb}:
    $$\frac{gcd(a_1,...,a_k)}{gcd(d,gcd(a_1,...,a_k))}=gcd(\frac{a_1}{gcd(d,a_1)},...,\frac{a_k}{gcd(d,a_k)}).$$

\begin{lem}\label{dd}

  Let $x\in G$, and let $d$ be  a  positive integer. Then  $E(x^d, N)= \frac{E(x, N)}{gcd(d, E(x, N))}$.
  In particular, $E(x^{E(x, N)}, N)=1$.

 \end{lem}
\begin{proof} Let $G=\{g_1,...,g_m\}$. We have
\begin{eqnarray*}
E(x^d,N)&=&gcd(n(x^{g_1})^{ d}, N),..., n(x^{g_m})^{d}, N))\\&=&
gcd(\frac{n(x^{g_1}, N)}{gcd(d, n(x^{g_1}, N))},...,\frac{n(x^{g_m}, N)}{gcd(d, n(x^{g_m}, N))})\\&=&
\frac{gcd(n(x^{g_1}, N),...,n(x^{g_m}, N))}{gcd(d,gcd(n(x^{g_1}, N),...,n(x^{g_m}, N)))}\\&=&
\frac{E(x, N)}{gcd(d,E(x, N))}.
\end{eqnarray*}

\end{proof}

\begin{cor}\label{hal}   Let $x\in G$,  $\pi$ a subset of all prime divisors of $|G|$.
 If all Hall $\pi$-subgroups of $G$ are conjugated and $N$ is a Hall $\pi$-subgroup of $G$, then $E(x,N)=o(x^{|N|})$.
 \end{cor}
\begin{proof}

   Let $x=yz=zy$ where all prime divisor of $o(y)$  belong to $\pi$ and   $o(z)=o(x^{|N|})$.
   We have  $E((yz)^{|N|},N)=E(z^{|N|},N)$.
   By Lemma  \ref{dd},
   \begin{eqnarray*}
 \frac{E(yz,N)}{gcd(E(yz, N), |N|)}&=& E((yz)^{|N|},N)\\&=&E(z^{|N|},N)\\&=&\frac{E(z, N)}{gcd(E(z, N), |N|)}\\&=&E(z, N)
   \end{eqnarray*}
    because
   $gcd(o(z), |N|)=1$.  Since $E(yz, N)\mid [G:N]$, we have $gcd(E(yz, N), |N|)=1$.
 Since  $\langle z^{u}\rangle \cap N=1$ for all $u\in G$, we have
$n(z^{gu}, N)=o(z^{u})=o(z)$. Then $E(x, N)=o(z)=o(x^{|N|})$.
\end{proof}

Let $G$ $d$ be a divisor of $[G:N]$. Let $L(d, N)=\{x\in G: E(x^d, N)=1\}$.
The following Proposition  shows that the value of $E(x, N)$ is depend on $o(x)$ and  Sylow subgroups of $N$.
\begin{prop}\label{equ}  Let  $|N|=p_1^{\alpha_1}...p_k^{\alpha_k}$, and let  $P_i\in Syl_{p_i}(N)$. Then

(i) for all $x\in G,$ we have  $$E(x, N)= gcd(E(x, P_1),...,E(x, P_k))$$

(ii) if $N$ is a $p$-group, then $L(d,N) = \{x \in G : x^d\in N^G\};$

(iii) $$L(d, N)=\bigcap_{i=1}^kL(d[N:P_i], P_i).$$
 \end{prop}
\begin{proof}
Let  $G=\{g_1,...,g_n\}$.

(i) Clearly, $E(x, N)\mid gcd(E(x, P_1),...,E(x, P_k))$.
We have $E(x, P_i)=E(x, N) r_i$ where $r_i=gcd(n(g_1^{-1}x^{E(x, N)}g_1, P_i),...,n(g_n^{-1}x^{E(x, N)}g_n, P_i)))$.  Since $r_i\mid [N:P_i]$, we have\break
$gcd(r_1,...,r_k)=1$. Then $gcd(E(x, P_1),...,E(x, P_k))\mid E(x, N)$.

(ii) Clearly, $\{u\in G: u^d\in N^G\} \subseteq L(d, N)$.
Let $x\in L(d, N)$. We have
$x=zy=yz$ where $o(z)=p^m$ for some integer $m$ and $p\nmid o(y)$.
We have $$E(x, N)=E(y, N)E(z,N)=o(y)E(z, N).$$
Also,  $E(z, N)=gcd(n(z^{g_1},N),...,n(z^{g_n}, N))$. By Lemma  \ref{min}, there exists
$g\in G$ such that $E(z, N)=n(z^g, N)$. It follows that $$E(x, N)=E(x^g, N)=o(y^g)n(z^g, N)=n(x^g, N).$$
Hence $E(x^d,N)=1$ if and only if $n((x^d)^g, N)=1$ for some $g\in G$.
Then
 $L(d,N)=\{u\in G: u^d\in N^G\}$.

(iii) If $N$ is a $p$-group, then there is nothing to prove.
 So let $k>1$.
Let $x\in \{u\in G: E(u^d, N)=1\}$. We have $E(x, P_i)=E(x, N)r_i$ where $r_i\mid [N:P_i]$.
Then $E(x^{d[N:P_i]}, P_i)=1$.
It follows that  $$L(d, N)\subseteq \bigcap_{i=1}^kL(d[N:P_i],P_i).$$

Now, let $x\in \bigcap_{i=1}^kL(d[N:P_i],P_i)$.
We have $E(x^{d[N:P_i]},P_i)=1$ for all $i=1,2,...,k$.
It follows that $$E(x^d, N)=gcd(E(x^d, P_1),...,E(x^d, P_k))\mid gcd([N:P_1],...,[N:P_k])=1.$$
Then $x\in L(d, N)$, and hence
$$\bigcap_{i=1}^kL(d[N:P_i],P_i)\subseteq L(d, N),$$
as claimed.
\end{proof}

Let     $V$ be a subset of $G$. We denote by
      $S(V,M,N)$ to be the set of all $x\in V$ such that $ E(x,M)=E(x,N)$.
    Write $Cor_G(N)=\bigcap_{g\in G}N^g$.
Then, with the Frobenius conjecture  in mind, the answer  to the following  question  would be interesting know.

\begin{que}Let  $d$ be a divisor of $[G:P]$ where $P\in Syl_p(G)$. Is it true that if 
$|L(d, N)|\leq d|N|$, then $L(d,N)$ is a normal  subgroup of $G$ and $|L(d, N)|=d|N|$.
\end{que}
Note that  by using the following  theorem which is a generalization of Frobenius Theorem, it is easy to show that $|Cor_G(N)|\cdot d \mid  |L(d,N)|$ whenever $d\mid[G:N]$. Also, it is not true $d|N|\mid |L(d,N)|$ in the general case.    

\begin{thm}[P. Hall \cite{hall}]\label{hall}
Let $G$ be a group of order $n$  and let $C$ be a class of $h$
conjugate elements. The number of solutions of $x^d = c$, where $c$ ranges
over $C$, is a multiple of $gcd(hd, n)$.
\end{thm}

Let $G$  be a finite    group. When $N\leq M$,
we have $E(x,M)\mid E(x, N)$. But in the case $N\nleq M$, we do not have such property. For this reason, we generalize our definition of order elements  to this case as follows.\begin{defn}
Let $G=\{g_1,...,g_n\}$  be a finite    group and $M_1,...,M_s$ be  subgroups of $G$.
We define $E(x, M_1,...,M_s)=gcd(E(x, M_1),..., E(x, M_s)))$.
\end{defn}

The following lemma is a direct result the properties in the second  section.

\begin{lem}\label{neq2}
Let $G=\{g_1,...,g_n\}$  be a finite    group and $M_1,...,M_s$ be  subgroups of $G$.
  Let $x,y,z\in G$.

  (i) If    $M_i\in Syl_{p_i}(G)$ for all $i=1,2,...,s$, then  $E(x, M_t, M_{t-1},...,M_1)=o(x^{|M_t| |M_{t-1}|...|M_1|})$.

  (ii) If $d$ is a positive integer, then
  $$E(x^d,M_1,...,M_s)= \frac{E(x,M_1,...,M_s)}{gcd(d, E(x,M_1,...,M_s))}.$$

  (iii) $E(x,1,M_1,...,M_s)=E(x,M_1,...,M_s)$.

 \end{lem}

\begin{rem}

Let $G$ be a finite group and $N$ a   subset of $G$ containing $1$.
If $N$ is a $G$-invariant subset of $G$ (i.e. $g^{-1}Ng\subseteq N$ for all $g\in G$) then 
$e(x,N)=n(x,N)$.

\end{rem}

\begin{lem}\label{1e} 
 
Let $G$ be a finite group and let  $L_d(G)=N$ and $L_q(G)=M$ such that $d\mid q$ and $q\mid |G|$.
Let   $x, y, z\in G$ and $m, k$  positive integers. Then

(i) if $x^m\in N$ and $x^k\in N$, then $x^{k-m}\in N$.

(ii) $n(x^m, N)=\frac{n(x, N)}{gcd(n(x, N), m)}$.

(iii) If $x=yz=zy$, then $n(yz, N)\mid lcm(n(y, N), n(z, N))$.
Also, $n(yz, N)= n(y, N)n(z, N)$  if and only if $gcd(n(y, N), n(z, N))=1$.

(iv)
$n(x, N)=n(x, M) n(x^{n(x, M)}, N)$.
 \end{lem}
\begin{proof}

(i) Let $x^m,x^k\in N$.
It follows from  $x^{dm}=1=x^{-dk}$ that $(x^{m-k})^d=1$, so $x^{m-k}\in N$.

(ii) First we prove that if $n(x^m, N)=1$, then $n(x, N)\mid m$.
There are positive integers $r$ and $q$ such that $m=qn(x,N)+r$ with $0\leq r<n(x,N)$.
Since $x^m$ and $x^{-q\cdot n(x, N)}\in N$, we have $x^r\in N$. Hence $r=0.$

Let $d=gcm(n(x, N), m)$. It follows from $(x^m)^{\frac{n(x, N)}{d}}=(x^{\frac{m}{d}})^{n(x, N)}\in N$
that $n(x^m, N)\mid \frac{n(x, N)}{d}$.
Since $(x^m)^{n(x^m, N)}\in N$, we conclude that
$n(x, N)\mid m\cdot n(x^m, N)$.
Therefore $\frac{n(x, N)}{d}\mid \frac{m}{d} n(x^m, N)$.
Since $gcd\Big( \frac{n(x, N)}{d}, \frac{m}{d}\Big)=1$, we have
 $\frac{n(x, N)}{d}\mid  n(x^m, N)$.

(iii) and (iv) the proofs  of these two part are the same as the proofs of part (iii) and (iv) of Lemma \ref{1}.

\end{proof}
\begin{lem}\label{dde}

  Let $ G$ be a finite group of order $n$, and let $d$ be  a divisor of $n$.  Let $x\in G$. Then  $e(x, L_d(G))= \frac{n(x, L_d(G))}{gcd(d, n(x, L_d(G)))}$.
  In particular, $e(x^{n(x, L_d(G))}, L_d(G))=1$.

 \end{lem}
\begin{proof} 
The proof of this  part are the same as the proof of Lemma \ref{dd}.
\end{proof}

We end this section by defining the third type of generalized element order. Albeit it will not be used in what follows, it is worth mentioning.
\begin{defn}\label{d12}
Let $G$ be a finite group and $N$  be  subgroup of $G$. For all $x\in G$ we define
$o_N(x)=gcd(E(xn_1,N),...,E(xn_k,N))$  where $N=\{n_1,...,n_k\}$.
\end{defn}
We will show that the set of all elements $x\in G$ with $o_N(G)=1$ form a normal subgroup of $G$.
The following lemma will be employed.
\begin{lem}\label{nuuum3} Let $G$  be a finite    group and $N$  a   subgroup  of $G$.
Let $x\in G$, and let  $y\in C_G(x)$.

(i) $o_N(x)\mid [G:N]$.

(ii) $o_N(x)=o_N(x^g)$ for all $g\in G$.

(iii) $o_N(x)=o_N(xn)$ for all $n\in N$.

(iv)If $gcd(n(x^u,N), n(y^v,N))=1$ for all $u,v\in G$, then
$o_N(xy)= o_N(x)o_N(y)$.

(v) for all positive integer $d$, we have $o_N(x^d)=\frac{o_N(x)}{gcd(o_N(x), d)}$.

(vi) for all $y\in N$ and $g\in G$, we have $o_N(xy^g)=o_N(x).$

 \end{lem}
\begin{proof} We only prove part (vi) as the  proofs of other parts are straightforward.
We have
$$o_N(x^{g^{-1}})=o_N(x^{g^{-1}}y)=o_N((xy^g)^{g^{-1}})=o_N(xy^g).$$

\end{proof}

\begin{thm}
Let $Ker(G,N)=\{x\in G: o_N(x)=1\}$. Then $$L(G,N)=\langle \{a^G: a\in N\}\rangle.$$

\end{thm}
\begin{proof}
 If $L(G,N)=N$, then clearly, $N\lhd G$. Therefore we may assume that $L(G,N)\neq N$.
Let $N^G=\{y_1,...,y_k\}$.
Let $z=a_1^{\epsilon_1}a_2^{\epsilon_2}...a_s^{\epsilon_s}$ where
$a_i\in L(G,N)$ and $\epsilon_i\in \{-1,1\}$.
By Lemma \ref{nuuum3}(vi), we have
$o_N(z)=o_N(za_s^{-\epsilon_s})=...=o_N(za_s^{-\epsilon_s}...a_1^{-\epsilon_1})=o_N(1)=1$.
It follows that $z\in L(G,N)$.
Consequently, $L=\langle N^G\rangle\leq Ker(G,N)$.
Since $L\lhd G$, for all $x\in G\setminus L$, we have
$o(xL)>1$. It follows that
$Ker(G,N)\leq L$.

\end{proof}



\section{two solvability criteria}\label{sec4}
In this section we shall exploit further the quantity $E(x,N)$ which will result in the proofs two solvability criteria. The first criterion will be derived in the spirit of the Cauchy's theorem and the second in the spirit of Hall subgroups. 

 Let   $p$ be a prime divisor of
 $[G:N]$.   Is it true that there exists $x\in G$ such that $E(x, N)=p$?
 In general, the answer is negative. This can be seen directly by the following example.

  \begin{exam}\label{exam}
Let $ G=PSL(2,q)$ where  $q = 2^n$  with $n>1$, and Let $F$ be a finite field of order $q$. Then $PSL(2, q)$
has a partition $\mathcal{P}$ consisting of $q +1$ Sylow $2$-subgroups, $(q+1)q$
cyclic subgroups of order $q-1 $ and $(q-1)q^2$
 cyclic subgroups of order $q+1$. Let $N=\langle \left(
                                                             \begin{array}{cc}
                                                               1 & 1 \\
                                                               0 & 1 \\
                                                             \end{array}
                                                           \right)
 \rangle$. Then $|N|=2$.
 It is clear that $$P=\{\left(
                                                             \begin{array}{cc}
                                                               1 & c \\
                                                               0 & 1 \\
                                                             \end{array}
                                                           \right) :   c\in F\}\in Syl_2(G).$$
  Let $c\in F^*\setminus\{1\}$.  Since $2\nmid |F^*|$, there is $u\in F^*$ such that $u^2=c$. Since   $$\left(
                                                             \begin{array}{cc}
                                                               u & 0 \\
                                                               0 & u^{-1} \\
                                                             \end{array}
                                                           \right) \left(
                                                             \begin{array}{cc}
                                                               1 & 1 \\
                                                               0 & 1 \\
                                                             \end{array}
                                                           \right)\left(
                                                             \begin{array}{cc}
                                                               u^{-1} & 0 \\
                                                               0 & u \\
                                                             \end{array}
                                                           \right) =\left(
                                                             \begin{array}{cc}
                                                               1 & c \\
                                                               0 & 1 \\
                                                             \end{array}
                                                           \right),$$
                                                           we have $E(x, N)=1$ for all $x\in P$, and so
                                                           $G$ does not have any element $x$ such that
                                                           $E(x, N)=2$. Clearly, $2^{n-1}\mid [G:N]$.
 \end{exam}
Luckily, in the case of a solvable group the answer of the aforementioned question is positive. Thus, in the set of all finite solvable groups we have the following theorem.

\begin{thm} \label{soli}Let $G$  be a finite  solvable  group and  $M$  be a     subgroup  of $G$. Let $p$ a prime divisor of
$[G:M]$. Then there exists $x\in G$ such that $E(x, M)=p$.
 \end{thm}

\begin{proof}   We proceed   by induction on $|G|$.
The base step is trivial when $|G|=p$. Clearly, we may assume that $|M|>1$. Let $N$ be a normal minimal subgroup of $G$.
Then $N$ is an elementary abelian $q$-group. First suppose that
$p\mid [\frac{G}{N}:\frac{MN}{N}]$. By induction hypothesis, there exists $xN\in \frac{G}{N}$ such that
$E(xN, \frac{MN}{N})=p$. Then $p\mid E(x, NM)$, and so $p\mid E(x, M)$. Let $y=x^{E(x, M)/p}$. Clearly,  $E(y,M)=p$, as claimed.
Therefore we may assume that  $p\nmid [\frac{G}{N}:\frac{MN}{N}]$. Since $p\mid [G:M]$, we have $p\mid |N|$.  It follows that $p=q$. Clearly,  $N $ is not a subgroup of $M$.
If $MN\neq G$, then we have the result by induction on $MN$. Hence we may assume that $MN=G$.
  If $G$ has another normal minimal subgroup $D$, then  $p\mid [\frac{G}{D}:\frac{MD}{D}]$, because $N\cap D=1$, so we have the result by induction hypothesis on $ \frac{G}{D}$.  Hence we may assume that $N$ is the unique normal minimal subgroup of $G$. Clearly, $Fit(G)$ is a $p$-group.  Let $P\in Syl_p(G)$. Since $p\nmid [G:N]$, we have
$N=P=Fit(G)$.
      Let $T$ be a maximal subgroup of $G$ containing $M$.
     If $T\neq M$, then we have the result, by induction hypothesis. Therefore we may assume that
     $T=M$. Let $g\in M$. We have $(N\cap M)^g\leq M$.
     Since $N\lhd G$, we deduce that $(N\cap M)^g\leq N$. It follows that
     $(N\cap M)^g\leq N\cap M$, so $(N\cap M)^g=N\cap M$. Consequently,
     $M\leq N_G(N\cap M)$. Clearly,   $ N\leq N_G(N\cap M)$, and so
     $N\cap M\lhd G$. Since $N$ is the unique normal minimal subgroup of $G$, we have
     $N\cap M=1$.
    Then for all $x\in N\setminus \{1\}$, we have $x^G\cap M=1$. It follows that $E(x, M)=o(x)=p$ for all $x\in N\setminus  \{1\}$.

\end{proof}
Now, to turn this latter theorem into a solvability criterion, we clearly need to make it into an "if and only if" statement.  Before we do so, however, the following Theorem must be made as it will be an important ingredient in the proofs of the solvability criteria below.
\begin{thm}\label{remsol}( Corollary 1 of \cite{tam})
If $G$ is a minimal non-solvable group, then by Corollary 1 of \cite{tam}, it is isomorphic to one of the following
groups:

\begin{enumerate}
\item  $PSL(2, q)$  where $q=p>3$ and $5\nmid p^2-1$.
  \item $PSL(2, q)$  where $q=2^p$ and $p$ is a prime.
  \item  $PSL(2, q)$ where $q=3^p$ and $p$ is an odd prime.
  \item $PSL(3, 3)$.
  \item   The Suzuki group $Sz(2^p)$, where $p$ is an odd prime.

\end{enumerate}

\end{thm}

\begin{thm}\label{sol1} A finite  group $G$ is solvable if and only if  for all subgroups  $N$  of $M$ with $p\mid [M:N]$ there exists
$x\in M$ such that $E(x, N)=p$.
 \end{thm}
\begin{proof}
If $G$ is solvable, then we have the result by Theorem \ref{soli}.
Let $G$ be defined as above so that it is  of minimal cardinality and unsolvable.
  Let $M$ be a maximal subgroup of $G$.
By the minimality of $G$ we have  $M$ is a solvable group.
It follows that $G$ is a minimal non-solvable group.
By Theorem \ref{remsol} we must do a case analysis.
By Example \ref{exam}, $G$ is not  isomorphic to $PSL(2, q)$  where $q=2^p$ and $p$ is a prime.
Let $G$ be isomorphic to $PSL(2, q)$  where $q>3$ is an odd order.
 By Theorem 7.3 \cite{gor}, $G$ has   one conjugate class of involution. By Theorem of \cite{gor}, page 462 Sylow $2$-subgroup of $G$ is Dihedral group,
 Let $P\in Syl_2(G)$. Then  there are $x,y\in P$ such that $P=\langle x, y\rangle$ where $o(x)=\frac{|P|}{2}$ and $o(y)=2$.
 Let $N=\langle x\rangle$ be a subgroup of $G$ of order two.
Hence $2\mid [G:N]$, but there is not an  element  $x\in G$ such that $E(x,N)=2$. Let $G\cong PSL(3,3)$. The GAP software \cite{Gap} yields that $G$ has   one conjugate class of involution and one conjugate class of elements  of order 4.
Also, $G$ has a cyclic  subgroup $H$ of order eight. Since $H$ is a normal cyclic subgroup of a Sylow $2-$subgroup of $G$, we have   $E(x, H)=2$ for all $x\in G$.
Let  $G$ is  the Suzuki group $Sz(2^p)$, where $p$ is an odd prime. Let $P\in Syl_2(G)$.
  The Suzuki groups have a single class of involution and   2 classes of elements of order 4. The exponent of a Sylow $2-$subgroup of $G$ is 4 and nilpotency classes of $P$ is 2.   Let $N$ be a maximal subgroup of $P$ with maximal number of elements of order 4.  Then
 $2\mid [G:N]$, and for  any  element   $x\in G$  we have  $E(x, N)\neq2$.
\end{proof}

As we have mentioned above, the second solvability criterion will see more of the structure of our group. We have the following theorem.
\begin{thm}\label{ma2}Let $G$  be a finite    group of order $n=p_1^{\alpha_1}p_2^{\alpha_2}...p_k^{\alpha_k}$ where $p_2, ..., p_k$ are primes.
   If for all  subgroups $M$ of $G$ there exists a set $\Gamma(M)$ of Sylow subgroups of $M$ with the following properties:

  (i) $\Gamma(M)\cap Syl_p(M)\neq \varnothing$ for all prime divisors $p$ of $|M|$;

  (ii) for all  $P_i,P_j\in \Gamma(M)$, and all $x\in M$, we have
$E(x, P_i,P_j)=e(x,P_iP_j)$.

 Then $G$ is a solvable group.
 \end{thm}

\begin{proof} Let $G$ be defined as above so that it is  of minimal cardinality and unsolvable.
  Let $M$ be a maximal subgroup of $G$.
Clearly, by the minimality of $G$, we have  $M$ is a solvable group.
It follows that $G$ is a minimal non-solvable group. By Theorem \ref{remsol} we consider the following three cases.

{\bf Case 1.} Let $G=PSL(2, q)$, and let $GF(q)$ the finite field with $q$ element. Let $c=gcd(2,q-1)$.
According to a well known theorem of \cite{di}( see chapter 12 pages 260-287), $G$ has a maximal subgroup of order $q(q-1)/c$  such as $W$.

Let $Q\in Syl_e(\langle x \rangle)$  such that $o(x)=q+1/c$ and $e\mid q+1/c$, and let $P\in Syl_{p}(G)\cap \{P_1,...,P_k\}$ such that $|P|=q$.
We may assume that $P\leq W$. By Theorem 3.21 \cite{ab}, $P\lhd W$, and the number of distinct conjugates of $P$ is $q+1$.
By Theorem 3.21 \cite{ab}, $N_G(\langle x \rangle)$   is a dihedral
group of order $q+1$.
Then   $\{P^{b}: b\in N_G(\langle x \rangle)\}$ is the set of all distinct conjugates   of $P$ in $G$. Also, $\{Q^y: y\in W\}$ is the set of all distinct conjugates  of $Q$ in $G$. Let
$\Theta=\{P^{b}Q^y: y\in W, \  b\in N_G(\langle x \rangle)\}$.  We claim that $\Theta\subseteq \{(PQ)^g: g\in G\}$.  Let
$g=by$ where $y\in W$ and $ b\in N_G(\langle x \rangle)$. We have $P^g=P^{b_2}$ for some integer $j$ and   $ b_2\in N_G(\langle x \rangle)$. Hence
$P^{g}Q^{g}=P^{b_2}Q^{y}\in\Theta$, and so  $\Theta\subseteq \{(PQ)^g: g\in G\}$, as claimed.

Let $\Phi=\bigcup_{g\in G}(P Q)^g $, and let  $\Delta=\bigcup_{g\in G}(P\cup Q)^g$.
First, suppose that $\Phi\neq \Delta$. Therefore there are
$u\in P$ and $v\in Q$ such that $(u v)^g\not\in \Delta$ for some $g\in G$.   Then $o(uv)\nmid |P||Q|$. Let $Q^y\in \{P_1,...,P_k\}$.
Since $PQ^y\in \Theta$, there exists $g\in G$ such that $PQ^y=(PQ)^g$.
Then $e(uv, PQ)=e(uv, PQ^y)=1$, and $E(uv, P,Q^y)=o((uv)^{|P||Q|})=o(uv)>1$, which is a contradiction.

So  let  $\Phi=\Delta$.
Clearly, $\Delta\neq G$.
If $\Delta\leq G$, then $\Delta=PQ$ is a normal subgroup, which is a contradiction.
So $\Delta\nleq G$. Then there exists $u\in P^{b}$ where $ b\in N_G(\langle x \rangle)$ and $v\in Q^y$ where $y\in W$,  such that
$uv\not\in \Delta$. It follows that $o(uv)\nmid |P||Q|$. Clearly, $P^{x^i}Q^y\in \Theta$, so there is $g\in G$ such that $P^{x^i}Q^y=(PQ)^g$. Then $e(uv, P^{x^i}Q^y)=e(uv, (PQ)^g)=e(uv, PQ)=1$, and $E(uv, P^{x^i},Q^y)=o((uv)^{|P||Q|})=o(uv)>1$, which is a contradiction.

{\bf Case 2.}
Let $G=PSL(3,3)$. Let $W$ be a subgroup of $G$ of order $432$. The number of distinct conjugates of $W$ in $G$ is equal $13$.   By using GAP software we see that, $o(g)\in \{ 1, 2, 3, 4, 6, 8, 13\}$. We may assume that $P\in Syl_2(W)\cap \{P_1,...,P_k\}$, and let $Q\in Syl_{13}(W)$. There exist $x\in G$ such that $Q^x\in \{P_1,P_2,P_3\}$. Since $|N_G(Q)|=39$, we have $\{Q^g: g\in P\}$ is the set of all conjugates of $Q$.
Then $\{PQ^t: t\in G\}=\{(PQ)^h: h\in Q\}$. By using GAP  software we see that, there are
  $u\in P$ and $v\in Q$, such that $o(uv)\mid 3$.
 Then $e(uv, P Q^x)=1$, and $E(uv, P , Q^x)=o((uv)^{|P||Q|})>1$, which is a contradiction.

{\bf Case 3.} Let $G$ be the Suzuki group $Sz(2^p)$, where $p$ is an odd prime. Let $P_1=P\in Syl_2(G)$. By Theorem 3.10 and 3.11 of \cite{hu}, the set $\Gamma=\{P^g, A^g, B^g, C^g: g\in G\}$
is a
partition of $G $ where $P$ is a Sylow 2-subgroup, $A$ is cyclic of order $q-1$, $B$ is cyclic
of order $q-2\sqrt{q/2}+1$ and $C$ is cyclic of order $q+2\sqrt{q/2}+1$.
Let $W$ be a maximal subgroup of $G$ containing $P$ of order $|P|(q-1)$. It is known that $P\lhd W$. Let $Q\in Syl_e(\langle x \rangle)$  such that $o(x)=q-1$ and $e\mid q-1$.  The rest of the proof is similar to the case 1 with a little change.
\end{proof}

In the spirit of     Theorem \ref{ma2} the following question might be interesting to ponder.
\begin{que}Let $G$  be a finite    group of order $n=p_1^{\alpha_1}p_2^{\alpha_2}...p_k^{\alpha_k}$ where $p_2, ..., p_k$ are primes.
Let $1\leq i<j \leq k$.
Suppose that for all  subgroups $M$ of $G$ there exists a set $\Gamma(M)=\{P_i,P_j\}$ of Sylow subgroups of $M$ such for   $x\in M$, we have
$E(x, P_i,P_j)=e(x,P_iP_j)=e(x,P_jP_i)$.
Is it true that  $G$ has a subgroup of order  $p_i^{\alpha_i}p_j^{\alpha_j}$?

 \end{que}


\bibliographystyle{amsplain}

\end{document}